\documentclass[graybox]{svmult}
\usepackage{amsmath}
\usepackage{amsfonts}
\usepackage{amssymb}
\usepackage{enumerate}
\usepackage{xcolor}

\usepackage{type1cm}         
\usepackage{makeidx}         
\usepackage{graphicx}        
\usepackage{multicol}        
\usepackage[bottom]{footmisc}
\usepackage{newtxtext}
\usepackage{newtxmath} 
\makeindex

\title*{An interesting class of Hankel determinants}
\author{Johann Cigler\thanks{Email: johann.cigler@univie.ac.at}\; \& Mike Tyson\thanks{Email: mgtyson66@gmail.com}}
\date{}

\theoremstyle{plain}
\theoremsymbol{}
\newtheorem{thm}{Theorem}[section]

\newtheorem{lem}[thm]{Lemma}
\newtheorem{defn}[thm]{Definition}
\newtheorem{prop}[thm]{Proposition}

\newcommand{\ZZ}{\mathbb{Z}}

\newcommand{\eps}{\varepsilon}
\DeclareMathOperator{\adj}{adj}
\DeclareMathOperator{\rk}{rk}
\begin{document}

\maketitle

\abstract{For small $r$ the Hankel determinants $d_r(n)$ of the sequence $\left({2n+r\choose n}\right)_{n\ge 0}$ are easy to guess and show an interesting modular pattern. For arbitrary $r$ and $n$ no closed formulae are known, but for each positive integer $r$ the special values $d_r(rn)$, $d_r(rn+1)$, and $d_r(rn+\lfloor\frac{r+1}{2}\rfloor)$ have nice values which will be proved in this paper.}

\section{Introduction}
\label{sec:1}

Let $(a_n)_{n\ge 0}$ be a sequence of real numbers with $a_0=1$. For each $n$ consider the Hankel determinant
\begin{equation}
H_n=\det(a_{i+j})_{i,j=0}^{n-1}.
\end{equation}
We are interested in  the sequence $(H_n)_{n\ge 0}$ for the sequences $a_{n,r}={2n+r\choose n}$ for some $r\in\mathbb{N}$. For $n=0$ we let $H_0=1$.

Let
\begin{equation}
d_r(n)=\det\left({2i+2j+r\choose i+j}\right)_{i,j=0}^{n-1}.
\end{equation}
For $r=0$ and $r=1$ these determinants are well known and satisfy $d_0(n)=2^{n-1}$ and $d_1(n)=1$ for $n>0$. E{\u g}ecio{\u g}lu, Redmond, and Ryavec \cite{EGECIOGLU201077} computed $d_2(n)$ and $d_3(n)$ and stated some conjectures for $r>3$.

Many of these determinants are easy to guess and show an interesting modular pattern. For example
\begin{align}
(d_0(n))_{n\ge 0}&=(1,1,2,2^2,2^3,\dots),\\
(d_1(n))_{n\ge 0}&=(1,1,1,1,1,\dots),\\
(d_2(n))_{n\ge 0}&=(1,1,\mathbin{\color{red}-1,-1,}1,1,\mathbin{\color{red}-1,-1,}\dots),\\
(d_3(n))_{n\ge 0}&=(1,1,-4,\mathbin{\color{red}3,3,-8,}5,5,-12,\mathbin{\color{red}7,7,-16,}\dots),\\
(d_4(n))_{n\ge 0}&=(1,1,-8,8,\mathbin{\color{red}1,1,-16,16,}1,1,-24,24,\dots),\\
(d_5(n))_{n\ge 0}&=(1,1,-13,-16,61,\mathbin{\color{red}9,9,-178,-64,370,}25,25,-695,-144,1127,\dots)
\end{align}
These and other computations suggest the following facts:
\begin{align}
&d_{2k+1}((2k+1)n)=d_{2k+1}((2k+1)n+1)=(2n+1)^k,\\
&d_{2k+1}((2k+1)n+k+1)=(-1)^{k+1\choose 2}4^k(n+1)^k,\\
&d_{2k}(2kn)=d_{2k}(2kn+1)=(-1)^{kn},\\
&d_{2k}(2kn+k)=-d_{2k}(2kn+k+1)=(-1)^{kn+{k\choose 2}}4^{k-1}(n+1)^{k-1}.
\end{align}
The purpose of this paper is to prove these conjectures. These methods seem to extend to the Hankel determinants of the sequences $\left({2n+r\choose n-s}\right)_{n\ge 0}$, but we do not compute these here.

In Sections 2 and 3 we review some well-known facts from the theory of Hankel determinants. In particular we compute $d_0(n)$ and $d_1(n)$. In Section 4 we define the matrix $\gamma$ and use it to compute $d_2(n)$. In Section 5 we introduce the matrices $\alpha_n$ and $\beta_n$, which serve as the basis of our method. In Section 6 we write the Hankel matrices in terms of these matrices. In Sections 7 and 8 we use this information to compute $d_r(n)$ in the aforementioned seven cases.

We would like to thank Darij Grinberg for his helpful suggestions.

\section{Some background material}

Let us first recall some well-known facts about Hankel determinants (cf. e.g. \cite{2018arXiv180105608C}). If $d_n=\det(a_{i+j})_{i,j=0}^{n-1}\neq 0$ for each $n$ we can define the polynomials
\begin{equation}
p_n(x)=\frac{1}{d_n}\det\begin{pmatrix}
a_0 & a_1 & \cdots & a_{n-1} & 1\\
a_1 & a_2 & \cdots & a_{n} & x\\
a_2 & a_3 & \cdots & a_{n+1} & x^2\\
\vdots & & & & \vdots\\
a_n & a_{n+1} & \cdots & a_{2n-1} & x^n\\
\end{pmatrix}.
\end{equation}

If we define a linear functional $L$ on the polynomials by  $L(x^n)=a_n$ then $L(p_np_m)=0$ for $n\neq m$ and $L(p_n^2)\neq 0$ (orthogonality).

By Favard's Theorem there exist complex numbers $s_n$ and $t_n$ such that
\begin{equation}
p_n(x)=(x-s_{n-1})p_{n-1}(x)-t_{n-2}p_{n-2}(x).
\end{equation}
For arbitrary $s_n$ and $t_n$ define numbers $a_n(j)$ by
\begin{align}\label{Fav1}
a_0(j)&=[j=0],\nonumber\\
a_n(0)&=s_0a_{n-1}(0)+t_0a_{n-1}(1),\\
a_n(j)&=a_{n-1}(j-1)+s_ja_{n-1}(j)+t_ja_{n-1}(j+1).\nonumber
\end{align}
These numbers satisfy
\begin{equation}\label{Fav2}\sum_{j=0}^n a_n(j)p_j(x)=x^n.\end{equation}

Let $A_n=(a_i(j))_{i,j=0}^{n-1}$ and $D_n$ be the diagonal matrix with entries $d(i,i)=\prod_{j=0}^{i-1} t_j$. Then we get
\begin{equation}\label{ADA}\left(a_{i+j}(0)\right)_{i,j=0}^{n-1}=A_nD_nA_n^\top\end{equation}
and
\begin{equation}\det\left(a_{i+j}(0)\right)_{i,j=0}^{n-1}=\prod_{i=1}^{n-1}\prod_{j=0}^{i-1}t_j.\end{equation}

If we start with the sequence $(a_n)_{n\ge 0}$ and guess $s_n$ and $t_n$ and if we also can guess $a_n(j)$ and show that $a_n(0)=a_n$ then all our guesses are correct and the Hankel determinant is given by the above formula.

There is a well-known equivalence with continued fractions, so-called J-fractions:
\begin{equation}\sum_{n\ge 0}a_nx^n=\cfrac{1}{1-s_0x-\cfrac{t_0x^2}{1-s_1x-\cfrac{t_1x^2}{1-\ddots}}}.\end{equation}
For some sequences this gives a simpler approach to Hankel determinants.

As is well known Hankel determinants are intimately connected with the Catalan numbers $C_n=\frac{1}{n+1}{2n\choose n}$. Consider for example the aerated sequence of Catalan numbers $(c_n)=(1,0,1,0,2,0,5,0,14,0,\dots)$ defined by $c_{2n}=C_n$ and $c_{2n+1}=0$. Since the generating function of the Catalan numbers
\begin{equation}C(x)=\sum_{n\ge 0}C_nx^n=\frac{1-\sqrt{1-4x}}{2x}\end{equation}
satisfies
\begin{equation}C(x)=1+xC(x)^2,\end{equation}
we get
\begin{equation}C(x)=\frac{1}{1-xC(x)}\end{equation}
and
\begin{equation}C(x^2)=\frac{1}{1-x^2C(x^2)}=\cfrac{1}{1-\cfrac{x^2}{1-\cfrac{x^2}{1-\ddots}}}\end{equation}
and therefore
\begin{equation}\det(c_{i+j})_{i,j=0}^{n-1}=1.\end{equation}

From $C(x)=1+xC(x)^2$ we get $C(x)^2=1+2xC(x)^2+x^2C(x)^4$ or
\begin{equation}\label{C(x)^2}C(x)^2=\frac{1}{1-2x-x^2C(x)^2}=\cfrac{1}{1-2x-\cfrac{x^2}{1-2x-\cfrac{x^2}{1-2x-\ddots}}}.\end{equation}
The generating function of the central binomial coefficients $B_n={2n\choose n}$ is
\begin{equation}B(x)=\sum_{n\ge 0}B_nx^n=\frac{1}{\sqrt{1-4x}}=\frac{1}{1-2xC(x)}=\frac{1}{1-2x-2x^2C(x)^2}.\end{equation}
Therefore by \eqref{C(x)^2} we get the J-fraction
\begin{equation}B(x)=\frac{1}{1-2x-2x^2C(x)^2}=\cfrac{1}{1-2x-\cfrac{2x^2}{1-2x-\cfrac{x^2}{1-2x-\cfrac{x^2}{1-2x-\ddots}}}}.\end{equation}
Thus the corresponding numbers $t_n$ are given by $t_0=2$ and $t_n=1$ for $n>0$ which implies $d_0(n)=2^{n-1}$ for $n\ge 1$.

Let us also consider the aerated sequence $(b_n)$ with $b_{2n}=B_n$ and $b_{2n+1}=0$. Here we get
\begin{equation}b(x)=B(x^2)=\frac{1}{1-2x^2C(x)^2}=\cfrac{1}{1-\cfrac{2x^2}{1-\cfrac{x^2}{1-\cfrac{x^2}{1-\ddots}}}}.\end{equation}
In this case $s_n=0$, $t_0=2$, and $t_n=1$ for $n>0$. Here we also get $\det(b_{i+j})_{i,j=0}^{n-1}=2^{n-1}$ for $n>0$. The corresponding orthogonal polynomials satisfy $p_0(x)=1$, $p_1(x)=x$, $p_2(x)=xp_1(x)-2$ and $p_n(x)=xp_{n-1}(x)-p_{n-2}(x)$ for $n>2$. The first terms are $1, x, x^2-2, x^3-3x,\dots$.

Now recall that the Lucas polynomials
\begin{equation}L_n(x)=\sum_{k=0}^{\lfloor\frac{n}{2}\rfloor}(-1)^k{n-k\choose k}\frac{n}{n-k}x^{n-2k}\end{equation}
for $n>0$ satisfy $L_n(x)=xL_{n-1}(x)-L_{n-2}(x)$ with initial values $L_0(x)=2$ and $L_1(x)=x$. The first terms are $2,x,x^2-2,x^3-3x,\dots$. Thus $p_n(x)=\bar L_n(x)$, where $\bar L_n(x)=L_n(x)$ for $n>0$ and $\bar L_0(x)=1$.

For the numbers $a_n(j)$ we get
\begin{align}
&a_{2n}(2j)={2n\choose n-j},\\
&a_{2n+1}(2j+1)={2n+1\choose n-j},
\end{align}
and $a_n(j)=0$ else. Equivalently $a_n(n-2j)={n\choose j}$ and $a_n(k)=0$ else.

For the proof it suffices to verify \eqref{Fav1} which reduces to the trivial identities ${2n\choose n}=2{2n-1 \choose n-1}$, ${2n\choose n-j}={2n-1\choose n-j}+{2n-1\choose n-1-j}$, and ${2n+1\choose n-j}={2n\choose n-j}+{2n\choose n-1-j}$. Identity \eqref{Fav2} reduces to
\begin{equation}\label{inv}\sum_{k=0}^{\lfloor\frac{n}{2}\rfloor}{n\choose k}\bar L_{n-2k}=x^n.\end{equation}

\section{Some well-known applications of these methods}

Now let us consider
\begin{equation}
d_1(n)=\det{2i+2j+1\choose i+j}.
\end{equation}
The generating function of the sequence ${2n+1\choose n}$ is
\begin{equation}
\sum_{n\ge 0}{2n+1\choose n}x^n=\frac{1}{2}\sum_{n\ge 0}{2n+2\choose n+1}x^n=\frac{1}{2x}\left(\frac{1}{\sqrt{1-4x}}-1\right)=\frac{C(x)}{\sqrt{1-4x}}.
\end{equation}
Now we have
\begin{align}
\sqrt{1-4x}&=1-2xC(x)=(C(x)-xC(x)^2)-2xC(x)=C(x)(1-2x-xC(x))\nonumber\\
&=C(x)(1-2x-x(1+xC(x)^2))=C(x)(1-3x-x^2C(x)^2).
\end{align}
Therefore
\begin{equation}
\frac{C(x)}{\sqrt{1-4x}}=\frac{1}{1-3x-x^2C(x)^2}=\cfrac{1}{1-3x-\cfrac{x^2}{1-2x-\cfrac{x^2}{1-2x-\cfrac{x^2}{1-2x-\ddots}}}}.
\end{equation}
The corresponding sequences $s_n,t_n$ are $s_0=3$, $s_n=2$ for $n>0$ and $t_n=1$. Thus $d_1(n)=1$. The corresponding $a_i(j)$ are $a_i(j)={2i+1\choose i-j}$.

To prove this we must verify \eqref{Fav1} which reduces to
\begin{align}
&{1\choose -j}=[j=0],\\
&{2n+1\choose n}=3{2n-1\choose n-1}+{2n-1\choose n-2},\\
&{2n+1\choose n-j}={2n-1\choose n-j}+2{2n-1\choose n-1-j}+{2n-1\choose n-2-j}.
\end{align}
The first line is clear. The right-hand side of the second line gives
\begin{align}
3{2n-1\choose n-1}+{2n-1\choose n-2}&=2{2n-1\choose n-1}+{2n\choose n-1}\\
&={2n\choose n}+{2n\choose n-1}={2n+1\choose n}.\nonumber
\end{align}
For the third line we get
\begin{equation}
{2n-1\choose n-j}+2{2n-1\choose n-1-j}+{2n-1\choose n-2-j}={2n\choose n-j}+{2n\choose n-j-1}={2n+1\choose n-j}.
\end{equation}
By \eqref{ADA} we see that with
\begin{equation}
A_n=\left({2i+1\choose i-j}\right)_{i,j=0}^{n-1}
\end{equation}
we get
\begin{equation}\label{anan}
A_nA_n^\top=\left({2i+2j+1\choose i+j}\right)_{i,j=0}^{n-1}.
\end{equation}

Let us give a direct proof of \eqref{anan}. Observe first that
\begin{equation}\label{suml}
\sum_{l=0}^{n-1}{2i+1\choose i-l}{2j+1\choose j-l}=\sum_{l=0}^{i}{2i+1\choose i-l}{2j+1\choose j-l}=\sum_{l=0}^{j}{2i+1\choose i-l}{2j+1\choose j-l}
\end{equation}
and that
\begin{align}
\sum_{l=0}^{i}{2i+1\choose i-l}{2j+1\choose j-l}&=\sum_{l=0}^{i}{2i+1\choose i-l}{2j+1\choose j+1+l}\\
&=\sum_{k=j+1}^{i+j+1}{2i+1\choose i+j+1-k}{2j+1\choose k}\nonumber
\end{align}
and
\begin{align}
\sum_{l=0}^{j}{2i+1\choose i-l}{2j+1\choose j-l}&=\sum_{l=0}^{j}{2i+1\choose i+1+l}{2j+1\choose j-l}\\
&=\sum_{k=0}^{j}{2i+1\choose i+j+1-k}{2j+1\choose k}.\nonumber
\end{align}
Adding the last two equations and using \eqref{suml} to rewrite the left-hand side, we obtain
\begin{align}
&2\sum_{l=0}^{n-1}{2i+1\choose i-l}{2j+1\choose j-l}\nonumber\\
&=\sum_{k=0}^{j}{2i+1\choose i+j+1-k}{2j+1\choose k}+\sum_{k=j+1}^{i+j+1}{2i+1\choose i+j+1-k}{2j+1\choose k}\\
&=\sum_{k=0}^{i+j+1}{2i+1\choose i+j+1-k}{2j+1\choose k}={2i+2j+2\choose i+j+1}=2{2i+2j+1\choose i+j}.\nonumber
\end{align}
Since $A_n$ is a triangle matrix whose diagonal elements are ${2i+1\choose i-i}=1$ we get $\det(A_nA_n^\top)=1$.

\section{A new method}

Fix $k>0$. Let us consider the determinants of the Hankel matrices $B_n(k)=\left({2i+2j+2\choose i+j+1-k}\right)_{i,j=0}^{n-1}$. These have already been computed in \cite{CIGLER2011144}, Theorem 21. There it is shown that
\begin{equation}\label{bijk}
\det(B_{km}(k))=(-1)^{{m\choose 2}k+m{k\choose 2}}
\end{equation}
and $\det(B_n(k))=0$ else.

\begin{defn}
Let $\gamma^{(k)}$ be the infinite matrix given by $\gamma^{(k)}_{ij}=1$ if $|i-j|=k$ or $i+j=k-1$ and $0$ elsewhere, with rows and columns indexed by $\ZZ_{\ge 0}$. Let us also consider the finite truncations $\gamma^{(k)}|_N$, where $A|_N$ denotes the submatrix consisting of the first $N$ rows and columns of a matrix $A$. We shall also write $\gamma^{(1)}=\gamma$ and $\gamma^{(k)}|_N=\gamma^{(k)}_N$.
\end{defn}

For example $\gamma^{(1)}_5$ and $\gamma^{(2)}_5$ are the following matrices:
\begin{align}
\gamma^{(1)}_5=\begin{pmatrix}
1&1&0&0&0\\
1&0&1&0&0\\
0&1&0&1&0\\
0&0&1&0&1\\
0&0&0&1&0
\end{pmatrix}&\qquad
\gamma^{(2)}_5=\begin{pmatrix}
0&1&1&0&0\\
1&0&0&1&0\\
1&0&0&0&1\\
0&1&0&0&0\\
0&0&1&0&0
\end{pmatrix}
\end{align}

An alternative description will be useful in this section and the next. Let $J_n$ be the exchange matrix with $1$'s on its antidiagonal and $0$'s elsewhere. Let $Q_n$ be the block matrix $\begin{pmatrix}J_n & I_n\end{pmatrix}$. Let $\sigma_n$ be the shift matrix with $(i,j)$ entry equal to $1$ if $j=i-1$ and $0$ otherwise. Then $\gamma_n^{(k)}=Q_n\sigma_{2n}^kQ_n^\top$.

\begin{thm}\label{angan}
\begin{equation}\label{ancnk}
A_n\gamma^{(k)}_nA_n^\top=B_n(k).
\end{equation}
\end{thm}
\begin{proof}
If we set $\gamma_n^{(0)}=2I_n$, where $I_n$ denotes the $n\times n$-identity matrix, then we already know that \eqref{ancnk} holds for $k=0$.

For the general case, we have $(A_nQ_n)_{ij}=\binom{2i+1}{i+j-(n-1)}$. Hence the $(i,j)$ entry of $A_nQ_n\sigma_{2n}Q_n^\top A_n^\top$ is
\begin{align}
\sum_{0\le r,s\le n-1}(A_nQ_n)_{ir}(\sigma^k)_{rs}(A_nQ_n)_{js}&=\sum_{0\le r,s\le n-1}\binom{2i+1}{i+r-(n-1)}\delta_{r-k,s}\binom{2j+1}{j+s-(n-1)}\nonumber\\
&=\sum_{r=n-1-i}^{n+i}\binom{2i+1}{i+r-(n-1)}\binom{2j+1}{j+r-k-(n-1)}\nonumber\\
&=\sum_{r'=0}^{2i+1}\binom{2i+1}{r'}\binom{2j+1}{j-i+r'-k}\\
&=\binom{2i+2j+2}{i+j+1-k}.\nonumber
\end{align}
The last identity follows from the Chu-Vandermonde formula.
\end{proof}

\begin{lem}\label{detckn}
\begin{align}
&\det(\gamma^{(k)}_{2kn})=(-1)^{kn}\\
&\det(\gamma^{(k)}_{2kn+k})=(-1)^{kn+{k\choose 2}}
\end{align}
and all other determinants $\det(\gamma^{(k)}_n)$ vanish.
\end{lem}
\begin{proof}
By the definition of a determinant we have
\begin{equation}
\det(a_{i,j})_{i,j=0}^{n-1}=\sum_\pi\text{sgn}(\pi)a_{0,\pi(0)}a_{1,\pi(1)}\cdots a_{n-1,\pi(n-1)}
\end{equation}
where $\pi$ runs over all permutations of the set $\{0,1,\dots,n-1\}$. The determinants of the matrices $\gamma^{(k)}_n$ either vanish or the sum over all permutations reduces to a single term $\text{sgn}\pi_n c(0,\pi_n(0),k)c(1,\pi_n(1),k)\cdots c(n-1,\pi_n(n-1),k)$.

Let us first consider $k=1$. The last row of $\gamma^{(1)}_n$ has only one non-vanishing element $c(n-1,n-2,1)$. Thus each $\pi$ which occurs in the determinant must satisfy $\pi(n-1)=n-2$. The next row from below contains two non-vanishing elements $c(n-2,n-3,1)$ and $c(n-2,n-1,1)$. The last element is the only element of the last column. Therefore we must have $\pi(n-2)=n-1$. The next row from below contains again two non-vanishing elements, $c(n-3,n-4)$ and $c(n-3,n-2)$. But since $n-2$ already occurs as image of $\pi$ we must have $\pi(n-3)=n-4$. Thus the situation has been reduced to $\gamma^{(1)}_{n-2}$. In order to apply induction we need the two initial cases $\gamma^{(1)}_1$ and $\gamma^{(1)}_2$.

For $n=1$ we get $\pi(0)=0$ and for $n=2$ $\pi(0)=1$ and $\pi(1)=0$ since
\begin{equation}
\gamma^{(1)}_2=\begin{pmatrix}1&1\\1&0\end{pmatrix}.
\end{equation}

If we write $\pi=\pi(0)\cdots\pi(n-1)$ we get in this way $\pi_1=0$, $\pi_2=10$, $\pi_3=021$, $\pi_4=1032$,$\dots$. This gives $\text{sgn}\pi_n=-\text{sgn}\pi_{n-2}$ and thus by induction $\det \gamma^{(1)}_n=(-1)^{n\choose 2}$, which agrees with \eqref{bijk}.

For general $k$ the situation is analogous. The last $k$ rows and columns contain only one non-vanishing element. This implies $\pi(n-j)=n-j-k$ and $\pi(n-j-k)=n-j$ for $1\le j\le k$ and $n\geq 2k$. Hence $\pi$ restricts to a permutation of $\{0,1,...,n-2k-1\}$. Thus the determinant can be reduced to $\gamma^{(k)}_{n-2k}$ and we get $\det \gamma^{(k)}_n=(-1)^k\det \gamma^{(k)}_{n-2k}$ if $n\ge 2k$.

For $n=k$ $\gamma^{(k)}_k$ reduces to the anti-diagonal and thus $\det \gamma^{(k)}_k=(-1)^{k\choose 2}$. For $0<n<k$ the first row of $\gamma^{(k)}_n$ vanishes and thus $\det \gamma^{(k)}_n=0$. For $k<n<2k$ there are two identical rows because $c(k-1,0,k)=c(k,0,k)=1$ and $c(k-1,j,k)=c(k,j,k)=0$ for $0<j<n$. Thus we see by induction that
\begin{align}
&\det(\gamma^{(k)}_{2kn})=(-1)^{kn}\\
&\det(\gamma^{(k)}_{2kn+k})=(-1)^{kn+{k\choose 2}}
\end{align}
and all other determinants vanish. This is the same as \eqref{bijk} because $(-1)^{{2n\choose 2}k+2n{k\choose 2}}=(-1)^{kn}$ and $(-1)^{{2n+1\choose 2}k+(2n+1){k\choose 2}}=(-1)^{kn+{k\choose 2}}$.
\end{proof}

\section{Two useful matrices}

For the finite matrices $\gamma_N=\gamma|_N$ we have $\gamma_N^k\neq \gamma^k|_N$. In order to compute $\gamma^k|_N$ in the realm of $N$-by-$N$-matrices we introduce the auxiliary matrices $\alpha^{(k)}_N$ and $\beta^{(k)}_N$.

Let $J_N$ be the exchange matrix with $1$'s on its antidiagonal and $0$'s elsewhere. Let $Q_N$ be the block matrix $\begin{pmatrix}J_N & I_N\end{pmatrix}$. Let $\sigma_N(\eps)$ be given by
\begin{equation}
(\sigma_N(\eps))_{ij}=\begin{cases}
1 & \text{if } i=j+1\\
\eps & \text{if } (i,j)=(0,N-1)\\
0 & \text{otherwise}.
\end{cases}
\end{equation}
Define $\alpha^{(k)}_N$ and $\beta^{(k)}_N$ as
\begin{equation}
Q_N\sigma_{2N}(\eps)^kQ_N^\top=\begin{cases}
\alpha_N^{(k)} & \text{if } \eps=1\\
\beta_N^{(k)} & \text{if } \eps=-1\\
\gamma_N^{(k)} & \text{if } \eps=0,
\end{cases}
\end{equation}
and the last line has been stated before. We shall again suppress the superscripts when $k=1$.

As a slight variation, consider the following infinite square matrices with rows and columns indexed by $\ZZ\setminus\{0\}=\{\dots,-2,-1,1,2,\dots\}$. Let $\bar I$ be the identity matrix and let $\bar J$ be the exchange matrix with $\bar J_{n,-n}=1$ for all $n$ and $0$ elsewhere. Let $\bar \sigma$ be given by $\bar \sigma_{n,n-1}=1$ and $0$ elsewhere. Define also the infinite rectangular matrix $\bar Q$ with rows indexed by $\ZZ_+=\{1,2,\dots\}$ and columns indexed by $\ZZ\setminus\{0\}$ by $\bar Q_{|n|,n}=1$ for $n\in\ZZ\setminus\{0\}$ and $0$ elsewhere. Note that $\gamma^{(k)}=\bar Q\bar \sigma^k\bar Q^T$, after shifting indices from $\ZZ_{\ge 0}$ to $\ZZ_+$.

\begin{thm}\label{mult}
When $\delta$ stands for either $\alpha_N$, $\beta_N$, or $\gamma$ one has $\delta^{(k)}=\delta\cdot\delta^{(k-1)}-\delta^{(k-2)}$ with initial values $\delta^{(1)}=\delta$ and $\delta^{(0)}=2$.
\end{thm}
\begin{proof}
For $\alpha_N$ and $\beta_N$, take $\sigma=\sigma_{2N}(\pm 1)$, $Q=Q_N$, $J=J_{2N}$, and $I=I_{2N}$. For $\gamma$, take $\sigma=\bar \sigma$, $Q=\bar Q$, $J=\bar J$, and $I=\bar I$. Note that in either case $Q^\top Q=I+J$, $\sigma J\sigma=J$, and $QJ=Q$. For $k\ge 2$,
\begin{align}
\delta\cdot\delta^{(k-1)}&=Q\sigma Q^\top Q\sigma^{k-1} Q^\top\nonumber\\
&=Q\sigma(I+J)\sigma^{k-1}Q^\top\\
&=Q\sigma^k Q^\top+Q(\sigma J\sigma)\sigma^{k-2} Q^\top\nonumber\\
&=\delta^{(k)}+\delta^{(k-2)}.\nonumber
\end{align}
\end{proof}

By induction we see that each $\gamma^{(k)}$ is a polynomial in $\gamma$. Therefore all $\gamma^{(k)}$ commute. Theorem \ref{mult} shows that the matrices $\gamma^{(k)}$ are Lucas polynomials in $\gamma$. More precisely
\begin{equation}\label{lkg}
\gamma^{(k)}=L_k(\gamma).
\end{equation}
By the same argument, $\alpha_N^{(k)}=L_k(\alpha_N)$ and $\beta_N^{(k)}=L_k(\beta_N)$.

\begin{thm}\label{abg}
For any polynomial $p$ with $\deg p\le 2N$, $\frac{p(\alpha_N)+p(\beta_N)}{2}=p( \gamma)|_N$.
\end{thm}
\begin{proof}
Note that $(L_0, \dots,L_{2N})$ is a basis of the vector space of degree at most $2N$ polynomials, since $\deg(L_k)=k$. Therefore it suffices to show that $(L_k(\alpha_N)+L_k(\beta_N))/2=L_k(\gamma)|_N$ for $k\le 2N$. To wit,
\begin{align}
(L_k(\alpha_N)+L_k(\beta_N))/2&=(\alpha_N^{(k)}+\beta_N^{(k)})/2\nonumber\\
&=Q_N(\sigma_{2N}(1)^k+\sigma_{2N}(-1)^k)Q_{N}^\top/2\nonumber\\
&=Q_N\sigma_{2N}(0)^kQ_{N}^\top\\
&=\gamma_N^{(k)}\nonumber\\
&=L_k(\gamma)|_N.\nonumber
\end{align}
\end{proof}

\section{Relating the determinant to the $\gamma$ matrices}\label{relate}

Let $a_n$, $b_n$, and $g_n$ be the characteristic polynomials of $\alpha_n$, $\beta_n$, and $\gamma_n$, respectively. By cofactor expansion along the last row we get $g_n(x)=xg_{n-1}(x)-g_{n-2}(x)$, $a_n(x)=g_n(x)-g_{n-1}(x)$, and $b_n(x)=g_n(x)+g_{n-1}(x)$. This plus the initial conditions of the $n=1$ and $2$ cases gives $b_n(x)=L_n(x)$,
\begin{equation}\label{gngk}
g_n(x)=\sum_{k=0}^n (-1)^{n-k}\bar L_k(x),
\end{equation}
and
\begin{equation}
a_n(x)=L_n(x)+2\sum_{k=0}^{n-1}(-1)^{n-k}\bar L_k(x).
\end{equation}
Here $\bar L_n(x)$ is the Lucas polynomial $L_n(x)$ except when $n=0$, in which case it is $1$.

By Theorem \ref{angan}, $A\phi(\gamma)A^\top$ is Hankel for all polynomials $\phi$. Here $A$ represents the infinite matrix $\left({2i+1\choose i-j}\right)_{i,j\ge 0}$ with finite truncations $A|_n=A_n$. This is because $\phi(x)$ can be expanded as a sum of Lucas polynomials $L_k(x)$, each of which gives a Hankel matrix. Moreover, multiplying the polynomial by $(x+2)$ shifts the Hankel matrix forward by $1$. It suffices to show this for $L_k(x)$. Recall that
\begin{equation}
AL_k(\gamma)A^\top=A\gamma^{(k)}A^\top=\left(\binom{2i+2j+2}{i+j+1-k}\right)_{i,j\ge 0}.
\end{equation}
Then by Theorem \ref{mult},
\begin{equation}
AL_k(\gamma)(\gamma+2)A^\top=A(\gamma^{(k-1)}+\gamma^{(k+1)}+2\gamma^{(k)})A^\top
\end{equation}
has $(i,j)$ entry
\begin{equation}
\binom{2i+2j+2}{i+j+2-k}+2\binom{2i+2j+2}{i+j+1-k}+\binom{2i+2j+2}{i+j-k}=\binom{2i+2j+4}{i+j+2-k}
\end{equation}
by Pascal's identity, which is the $(i+1,j)$ entry of the original matrix.

We'll now write the Hankel matrices of the sequence $(\binom{2n+r}{n})_{n\ge 0}$ in terms of the $\gamma$ matrices and $A$. By the above results, when $i+j=n$ we have
\begin{equation}
(Ab_k(\gamma)A^\top)_{ij}=(AL_k(\gamma)A^\top)_{ij}=\binom{2n+2}{n+1-k}
\end{equation}
and
\begin{equation}
(Ab_k(\gamma)(\gamma+2)^{k-1}A^\top)_{ij}=\binom{2(n+k-1)+2}{(n+k-1)+1-k}=\binom{2n+2k}n.
\end{equation}

By induction on $k$, we'll show $(Ag_k(\gamma)A^\top)_{ij}=\binom{2n+1}{n-k}$. The $k=0$ case is \eqref{anan}. For $k\ge 1$,
\begin{align}
(Ag_k(\gamma)A^\top)_{ij}&=(A(b_k(\gamma)-g_{k-1}(\gamma))A^\top)_{ij}\nonumber\\
&=\binom{2n+2}{n+1-k}-\binom{2n+1}{n-(k-1)}=\binom{2n+1}{n-k}.
\end{align}
Hence
\begin{equation}
(Ag_k(\gamma)(\gamma+2)^kA^\top)_{ij}=\binom{2(n+k)+1}{(n+k)-k}=\binom{2n+2k+1}{n}.
\end{equation}
We've proven the following theorem.

\begin{thm}\label{Gamma}
For $r\ge 1$, let $k=\lfloor\frac{r}{2}\rfloor$ and $l=\lfloor\frac{r-1}{2}\rfloor$, and define the function
\begin{equation}
h_r(x)=
\begin{cases}
g_k(x) & \text{if } r=2k+1\\
b_k(x) & \text{if } r=2k.
\end{cases}
\end{equation}
For $N\ge k+l$, by Theorem \ref{abg}, $d_r(N)$ equals
\begin{equation}
\det\left(h_r(\gamma)(\gamma+2)^l|_N\right)=\det\left(\frac{1}{2}\left(h_r(\alpha_N)(\alpha_N+2)^l+h_r(\beta_N)(\beta_N+2)^l\right)\right).
\end{equation}
\end{thm}

\section{Structure of the matrices}\label{struct}

In this section we determine the structure of the matrices $(\beta_N+2)^{-1}$, $g_k(\alpha_N)$, $g_k(\beta_N)$, $b_k(\alpha_N)$, and $b_k(\beta_N)$, as well as the determinants of $g_k(\gamma)|_N$ and $b_k(\gamma)|_N$.

To determine $p(\alpha_N)$ and $p(\beta_N)$ for a polynomial $p$ of degree less than $N$, we begin by writing $p(\gamma)$ as a sum of $\gamma^{(k)}$ matrices using the multiplicative formula of Theorem \ref{mult}. We then apply Prop \ref{corner} to show that $p(\alpha_N)$ and $p(\beta_N)$ are the same as $p(\gamma)|_N$ on and above the anti-diagonal. The structure of $p(\alpha_N)$ follows from the symmetry of $\alpha_N$ across its anti-diagonal. The structure of $p(\beta_N)$ can be computed from $p(\alpha_N)$ and $p(\gamma)|_N$ with Theorem \ref{abg}.

\begin{prop}\label{block}
The determinant of a block matrix
\begin{equation}
\begin{pmatrix}
A&B\\
C&D
\end{pmatrix}
\end{equation}
where $A$ and $D$ are square and $D$ is invertible is $\det(D)\det(A-BD^{-1}C)$.
\end{prop}
\begin{proof}
Note that
\begin{equation}
\begin{pmatrix}
A&B\\
C&D
\end{pmatrix}
\begin{pmatrix}
I&0\\
-D^{-1}C&I
\end{pmatrix}
=\begin{pmatrix}
A-BD^{-1}C&B\\
0&D
\end{pmatrix},
\end{equation}
and that the determinant of a block-triangular matrix is the product of the determinants of its diagonal blocks.
\end{proof}

\begin{prop}\label{corner}
Let $T$ be a $N$-by-$N$ tridiagonal matrix and let $p$ be a polynomial of degree $d$. Let $v$ be the $N$-by-$1$ column vector with a $1$ in its last entry and $0$ elsewhere. Then the $(i,j)$ entries of $p(T)$ and $p(T+vv^\top)$ agree when $i+j\le 2(N-1)-d$.
\end{prop}
\begin{proof}
It suffices to prove this for $p(x)=x^d$. Call a $N$-by-$N$ matrix ``$k$-small" iff its entries $(i,j)$ with $i+j\le 2(N-1)-k$ are all $0$. For instance, $vv^\top$ is $1$-small.

Suppose a matrix $M$ is $k$-small. For $i+j\le 2(N-1)-k-1$, the $(i,j)$ entry of $TM$ is $\sum_{l=0}^{N-1}T_{il}M_{lj}=T_{i,i-1}M_{i-1,j}+T_{i,i}M_{i,j}+T_{i,i+1}M_{i+1,j}$. Since $M$ is $k$-small, its $(i-1,j)$, $(i,j)$, and $(i+1,j)$ entries are $0$, which implies that $TM$ is $(k+1)$-small. Similarly, $MT$, $vv^\top M$, and $Mvv^\top$ are $(k+1)$-small.

Consider $(T+vv^\top)^d-T^d$. Expanding the binomial product yields $2^d-1$ terms, all of which are products of $d$ $T$'s and $vv^\top$'s and contain at least one $vv^\top$. It follows from the above that each of these terms is $d$-small, so $p(T+vv^\top)-p(T)$ is $d$-small.
\end{proof}

\begin{lem}\label{B}
The inverse of $(\beta_N+2)$ is $(\frac{1}{2}(-1)^{i+j}(2\min\{i,j\}+1))_{i,j=0}^{N-1}$. The determinant of $(\beta_N+2)$ is $2$. For example,
\begin{equation}
(\beta_5+2)^{-1}=\frac{1}{2}\begin{pmatrix}
\phantom{-}1&-1&\phantom{-}1&-1&\phantom{-}1\\
-1&\phantom{-}3&-3&\phantom{-}3&-3\\
\phantom{-}1&-3&\phantom{-}5&-5&\phantom{-}5\\
-1&\phantom{-}3&-5&\phantom{-}7&-7\\
\phantom{-}1&-3&\phantom{-}5&-7&\phantom{-}9
\end{pmatrix}.
\end{equation}
\end{lem}
\begin{proof}
For $i\neq 0, N-1$ the row $i$ of $(\beta_N+2)$ is $(2\delta_{il}+\delta_{i,l-1}+\delta_{i,l+1})_{l=0}^{N-1}$. The product of this with column $j$ of the claimed inverse is
\begin{equation}
\begin{split}
& \sum_{l=0}^{N-1}(2\delta_{il}+\delta_{i,l-1}+\delta_{i,l+1})\frac{1}{2}(-1)^{l+j}(2\min\{l,j\}+1)\\
& =\frac{1}{2}(-1)^{i+j}(4\min\{i,j\}+2-2\min\{i+1,j\}-1-2\min\{i-1,j\}-1)\\
& =(-1)^{i+j}(2\min\{i,j\}-\min\{i+1,j\}-\min\{i-1,j\}).
\end{split}
\end{equation}
This is $0$ if $i+1\le j$ or $i-1\ge j$ and is $1$ if $i=j$.

The first row of $(\beta_N+2)$ is $(3,1,0,\dots,0)$, and the last row is $(0,\dots,0,1,1)$. Column $j\neq 0,N-1$ of the claimed inverse begins and ends as
\begin{equation}
\frac{1}{2}((-1)^j,(-1)^{j+1}3,\dots,(-1)^{j+N-2}(2j+1),(-1)^{j+N-1}(2j+1)),
\end{equation}
so it kills the first and last rows of $(\beta_N+2)$. Column $0$ of the claimed inverse begins and ends as $\frac{1}{2}(1,-1,\dots,(-1)^{N-2},(-1)^{N-1})$ while column $N-1$ begins and ends as $\frac{1}{2}((-1)^{N-1},(-1)^N 3,\dots,-(2N-3),2N-1)$. It's easy to verify that these columns have the correct products with rows of $(\beta_N+2)$.

The determinant $\det(\beta+2)$ is $(-1)^Nb_N(-2)$, which can be computed with recurrence in Section \ref{relate} to be $2$.
\end{proof}

\begin{lem}\label{gka}
For $k<N$, the $(i,j)$ entry of $g_k(\alpha_N)$ is $(-1)^{i+j+k}$ if $k\le i+j\le 2N-k-2$ and $|i-j|\le k$ and is $0$ otherwise. The $(i,j)$ entry of $g_k(\beta_N)$ is $(-1)^{i+j+k}$ if $k\le i+j\le 2N-k-2$ and $|i-j|\le k$, is $2(-1)^{i+j+k}$ if $2N-k-1\le i+j$, and is $0$ otherwise. For example,
\begin{equation}
g_2(\beta_6)=\begin{pmatrix}
\phantom{-}0&\phantom{-}0&\phantom{-}1&\phantom{-}0&\phantom{-}0&\phantom{-}0\\
\phantom{-}0&\phantom{-}1&-1&\phantom{-}1&\phantom{-}0&\phantom{-}0\\
\phantom{-}1&-1&\phantom{-}1&-1&\phantom{-}1&\phantom{-}0\\
\phantom{-}0&\phantom{-}1&-1&\phantom{-}1&-1&\phantom{-}1\\
\phantom{-}0&\phantom{-}0&\phantom{-}1&-1&\phantom{-}1&-2\\
\phantom{-}0&\phantom{-}0&\phantom{-}0&\phantom{-}1&-2&\phantom{-}2
\end{pmatrix}.
\end{equation}
\end{lem}
\begin{proof}
Recall that $g_j(\gamma)=\gamma^{(j)}-\gamma^{(j-1)}+\cdots\pm\gamma^{(1)}\mp 1$, by \eqref{gngk}. Therefore $\frac{1}{2}(g_k(\alpha_N)+g_k(\beta_N))=g_k(\gamma)|_N=\gamma_N^{(k)}-\gamma_N^{(k-1)}+\cdots\pm\gamma_N^{(1)}\mp 1$. From the definition of the $\gamma_N^{(j)}$, the $(i,j)$ entry of $g_k(\gamma)|_N$ is $(-1)^{i+j+k}$ if $k\le i+j$ and $|i-j|\le k$ and is $0$ otherwise.

Note that polynomials in $\alpha_N$ are symmetric about their anti-diagonal. Since the degree of $g_k$ is $k<N$, Prop \ref{corner} says that $g_k(\alpha_N)$ agrees with $g_k(\gamma)|_N$ on and above its anti-diagonal. Thus, the $(i,j)$ entry of $g_k(\alpha_N)$ is $(-1)^{i+j+k}$ if $k\le i+j\le 2N-k-2$ and $|i-j|\le k$ and is $0$ otherwise. Similarly, the $(i,j)$ entry of $g_k(\beta_N)=2g_k(\gamma)|_N-g_k(\alpha_N)$ is $(-1)^{i+j+k}$ if $k\le i+j\le 2N-k-2$ and $|i-j|\le k$, $2(-1)^{i+j+k}$ if $2N-k-1\le i+j$, and $0$ otherwise.
\end{proof}

\vspace{2 pt}

\begin{lem}\label{detgka}
\begin{equation}
\det g_k(\gamma)|_N=\begin{cases}
1&\text{if } N=(2k+1)n\\
(-1)^{k+1\choose 2}&\text{if } N=(2k+1)n+k+1\\
0&\text{otherwise.}
\end{cases}
\end{equation}
\end{lem}
\begin{proof}
When $N=0$ the determinant is vacuously $1$. When $0<N<k+1$, the first column is $0$. When $N=k+1$ the matrix is $0$ above its antidiagonal and $1$ on its antidiagonal, so its determinant is $(-1)^{k+1\choose 2}$. When $k+1<N<2k+1$, columns $k-1$ and $k+1$ are equal. Thus the claim holds for all $N<2k+1$. We'll show that for $N\ge 2k+1$, $\det g_k(\gamma)|_N=\det g_k(\gamma)|_{N-2k-1}$.

Fix $N\ge 2k+1$ and let $M=g_k(\gamma)|_N$. Subdivide $M$ into a block matrix consisting of the leading principal order-$N-1$ submatrix $M_{11}$, the bottom-right entry $M_{22}$, and the remainders of the last column and row $M_{12}$ and $M_{21}$. The determinant of $M$ is $\det(M_{22})\det(M')$, where $M'$ is the $N-1$-by-$N-1$ matrix $M_{11}-M_{12}M_{22}^{-1}M_{21}$ by Proposition \ref{block}.

We will perform cofactor expansion in the bottom right of $M'$. Since $M_{22}=(-1)^k$, the bottom right $k$-by-$k$ submatrix of $M'$ is the zero matrix. As a result, the only entry in the bottom row of $M'$ is the $1$ at $(N-2,N-k-2)$. After deleting its row and column, the only entry in the bottom row of $M'$ is the $1$ at $(N-3,N-k-3)$. This pattern continues up to the $1$ at $(N-k-1,N-2k-1)$. Since $M'$ is symmetric, a similar sequence of lone $1$'s can be removed in the last $k$ columns.

After the last $2k$ rows and columns have been removed, $M'$ has been reduced to $g_k(\gamma)|_{N-2k-1}$. The $2k$ removed $1$'s contribute a factor of $(-1)^k$ to the determinant, which comes from the parity of the permutation $(0\;k)(1\; k+1)\cdots(k-1\;2k)$. This cancels with the sign of $M_{22}$.
\end{proof}

\begin{lem}\label{detbka}
For $k<N$, the $(i,j)$ entry of $b_k(\alpha_N)$ is $1$ if $|i-j|=k$, $i+j=k-1$, or $i+j=2(N-1)-(k-1)$ and is $0$ otherwise. The $(i,j)$ entry of $b_k(\beta_N)$ is $1$ if $|i-j|=k$ or $i+j=k-1$, is $-1$ if $i+j=2(N-1)-(k-1)$, and is $0$ otherwise. In particular $b_k(\gamma)=\gamma^{(k)}$. Moreover,
\begin{equation}\det b_k(\gamma)|_N=\begin{cases}
(-1)^{kn}&\text{if } N=2kn\\
(-1)^{kn+{k\choose 2}}&\text{if } N=2kn+k\\
0&\text{otherwise.}
\end{cases}\end{equation}
\end{lem}
\begin{proof}
The first set of claims follow from the Lemma \ref{gka} and the fact that $b_k(x)=g_k(x)+g_{k-1}(x)$. The determinant of $\gamma^{(k)}$ was calculated in Lemma \ref{detckn}.
\end{proof}

\section{Calculation of the determinant}

In this section we prove the seven formulas mentioned in the introduction. Recall Theorem \ref{Gamma} and its notation.

Let $\mu_i=\frac{1}{2}((\alpha_N+2)^ih_r(\alpha_N)+(\beta_N+2)^ih_r(\beta_N))$ for $0\le i\le l$. From here on we'll suppress the subscripts on $\alpha_N$ and $\beta_N$. By Theorem \ref{Gamma}, we're interested in calculating $d_{r}(N)=\det\mu_l$. Note that
\begin{equation}\label{mui}
\mu_{i+1}=\mu_i(\beta+2)+ (\alpha+2)^ih_r(\alpha)vv^\top.
\end{equation}
The results of the previous section give us control over $\mu_0$. We will induct on the above equation to screw the smoothing operators $\alpha+2$ and $\beta+2$ into place, using the matrix determinant lemma to keep track of the determinants. In the seven cases proven here, the determinant or adjugate of $\mu_i$ is multiplied by a constant factor at each step.

\begin{prop}[Matrix determinant lemma]
If $A$ is an $n$-by-$n$ matrix and $u$ and $v$ are $n$-by-$1$ column vectors, then
\begin{equation}
\det(A+uv^\top)=\det(A)+v^\top\adj(A)u.
\end{equation}
\end{prop}
\begin{proof}
This is a polynomial identity in the entries of $A$, $u$, and $v$, so it suffices to prove it for the dense subset where $A$ is invertible. Consider
\begin{equation}
\begin{pmatrix}I&0\\v^\top&1\end{pmatrix}\begin{pmatrix}I+A^{-1}uv^\top&u\\0&1\end{pmatrix}\begin{pmatrix}I&0\\-v^\top&1\end{pmatrix}=\begin{pmatrix}I&u\\0&1+v^\top A^{-1}u\end{pmatrix},
\end{equation}
which shows that $1\cdot\det(I+A^{-1}uv^\top)\cdot 1=\det(1+v^\top A^{-1}u)$. Multiplying through by $\det A$ yields $\det(A+uv^\top)=\det(A)(1+v^\top A^{-1} u)=\det(A)+v^\top\adj(A)u$.
\end{proof}

\subsection{The case that $\mu_0$ is invertible}

\begin{lem}\label{main1}
Suppose there is an $N$-dimensional column vector $w$ such that $\mu_0w=h_r(\alpha_N)v$ and that the last $l-1$ entries of $h_r(\beta_N)w$ are $0$. Then
\begin{equation}
\det(\mu_l)=\det(\mu_0)2^l\left(1+v^\top(\beta_N+2)^{-1}w\right)^l.
\end{equation}
\end{lem}
\begin{proof}
By Prop \ref{corner}, $(\alpha+2)^i$ and $(\beta+2)^i$ differ only in the last $i$ columns. It follows from the second hypothesis that $(\beta+2)^ih_r(\beta)w=(\alpha+2)^ih_r(\beta)w$ for $0\le i<l$. Thus
\begin{equation}
\mu_iw=(\alpha+2)^ih_r(\alpha)v
\end{equation}
and
\begin{equation}
\det(\mu_i)w=\adj(\mu_i)(\alpha+2)^ih_r(\alpha)v
\end{equation}
for $0\le i<l$. By \eqref{mui} and the matrix determinant lemma,
\begin{align}
\det(\mu_{i+1})&=\det(\beta+2)\left(\det(\mu_i)+v^\top(\beta+2)^{-1}\adj(\mu_i)(\alpha+2)^ih_r(\alpha)v\right)\\
&=\det(\beta+2)\left(\det(\mu_i)+v^\top(\beta+2)^{-1}\det(\mu_i)w\right).\nonumber
\end{align}
Hence
\begin{equation}
\det(\mu_{i+1})=2\det(\mu_i)\left(1+v^\top(\beta_N+2)^{-1}w\right).
\end{equation}
\end{proof}

\begin{thm}
\begin{align}
&d_{2k+1}((2k+1)n)=(2n+1)^k\\
&d_{2k+1}((2k+1)n+k+1)=(-1)^{k+1\choose 2}4^k(n+1)^k\\
&d_{2k}(2kn)=(-1)^{kn}\\
&d_{2k}(2kn+k)=(-1)^{kn+{k\choose 2}}4^{k-1}(n+1)^{k-1}
\end{align}
\end{thm}
\begin{proof}
Given $w$, it is straightforward to verify the hypotheses and evaluate the final expression of Lemma \ref{main1} with the lemmas of Section \ref{struct}. For the first formula, take $w$ to be the $(2k+1)n$-dimensional column vector
\begin{equation}
w_1=(-1)^{n-1}\left(\sum_{m=0}^{n-1} (-1)^me_{(2k+1)m}-\sum_{m=0}^{n-1} (-1)^me_{(2k+1)m+2k}\right)+e_{N-1},
\end{equation}
where $\{e_i\}_{i=0}^{N-1}$ is the standard basis. Then $g_k(\alpha)w_1=g_k(\beta)w_1=e_{N-k-1}$.

For the second formula, take $w$ to be the $(2k+1)n+k+1$-dimensional column vector
\begin{equation}
w_2=(-1)^{n}\left(\sum_{m=0}^{n} (-1)^me_{(2k+1)m+k-1}-\sum_{m=0}^{n-1} (-1)^me_{(2k+1)m+k+1}\right)+e_{N-1},
\end{equation}
which gives $g_k(\alpha)w_2=e_{N-k-1}+e_{N-k}$ and $g_k(\beta)w_2=e_{N-k-1}-e_{N-k}$.

For the third formula, take $w$ to be the $2kn$-dimensional column vector
\begin{equation}
w_3=(-1)^{n-1}\left(\sum_{m=0}^{n-1} (-1)^me_{2km}-\sum_{m=0}^{n-1} (-1)^me_{2km+2k-1}\right)+e_{N-1},
\end{equation}
which gives $b_k(\alpha)w_3=b_k(\beta)w_3=e_{N-k-1}+e_{N-k}$.

For the fourth formula, take $w$ to be the $2kn+k$-dimensional column vector
\begin{equation}
w_4=(-1)^{n}\left(\sum_{m=0}^{n} (-1)^me_{2km+k-1}-\sum_{m=0}^{n-1} (-1)^me_{2km+k+1}\right)+e_{N-1},
\end{equation}
which gives $b_k(\alpha)w_4=e_{N-k-1}+3e_{N-k}$ and $b_k(\beta)w_4=e_{N-k-1}-e_{N-k}$.
\end{proof}

\subsection{The case that $\mu_0$ is singular}

We will make use of the following fact about the adjugate matrix.

\begin{prop}\label{adj}
The rank of the adjugate $\adj(M)$ of an $n$-by-$n$ matrix $M$ satisfies
\begin{equation}
\rk\adj(M)=\begin{cases}
n&\text{if } \rk M=n\\
1&\text{if } \rk M=n-1\\
0&\text{otherwise}
\end{cases}
\end{equation}
\end{prop}
\begin{proof}
Recall that $\adj(M)\cdot M=\det(M)I$. If $\rk M=n$ then $M$ is invertible with inverse $\frac{1}{\det(M)}\adj(M)$, which also has rank $n$.

If $\rk M=n-1$, then $\det(M)=0$, in which case $\adj(M)$ must send all vectors into the kernel of $M$, which has rank $1$. In this case $M$ also has a nonzero order-$n-1$ minor, so $\adj(M)$ has rank $1$.

If $\rk M\le n-2$, then all order-$n-1$ minors of $M$ are zero, so $\adj(M)=0$.
\end{proof}

\begin{lem}\label{main2}
Suppose there is a nonzero $N$-dimensional column vector $w$ such that $\det(\mu_0)=0$, $\det(\mu_0|_{N-1})\neq 0$, $\mu_0w=0$, $v^\top w=1$, $v^\top(\beta+2)^{-1}w\neq 0$, and entries $N-k-l$ through $N-3$ of $w$ are $0$. Then
\begin{equation}
\det(\mu_l)=\det(\mu_0|_{N-1})\left(2v^\top(\beta_N+2)^{-1}w\right)^l\left(w^\top(\alpha+2)^{l-1}h_r(\alpha)v\right).
\end{equation}
\end{lem}
\begin{proof}
Let $c=\det(\mu_0|_{N-1})$. We will show by induction that
\begin{equation}
\adj(\mu_i)=c\left(2v^\top(\beta_N+2)^{-1}w\right)^iww^\top,
\end{equation}
for $0\le i<l$. For the base case of $i=0$, note that the first two hypotheses imply that $\mu_0$ has rank $N-1$. Since $w$ generates the kernel and $\mu_0$ is symmetric, Proposition \ref{adj} implies that $\adj(\mu_0)$ is a constant $d$ times $ww^\top$. In fact $c=v^\top\adj(\mu_0)v=dv^\top ww^\top v=d$.

Suppose the claim holds for $i$. Since $\alpha+2$ is tridiagonal, the last hypothesis combined with Lemmas \ref{gka} and \ref{detbka} imply that $w^\top(\alpha+2)^ih_r(\alpha)v=0$. By \eqref{mui} and the matrix determinant lemma,
\begin{align}
\det(\mu_{i+1})&=\det(\beta+2)\left(\det(\mu_i)+v^\top(\beta+2)^{-1}\adj(\mu_i)(\alpha+2)^ih_r(\alpha)v\right)\\
&=\det(\beta+2)\left(0+c\left(2v^\top(\beta_N+2)^{-1}w\right)^iv^\top(\beta+2)^{-1}ww^\top(\alpha+2)^ih_r(\alpha)v\right)\nonumber\\
&=0,\nonumber
\end{align}
so $\mu_{i+1}$ has rank at most $n-1$. Since $(\alpha+2)^ih_r(\alpha)vv^\top$ doesn't affect the bottom-right cofactor,
\begin{align}
v^\top\adj(\mu_{i+1})v&=v^\top\adj\left(\mu_i(\beta+2)+ (\alpha+2)^ih_r(\alpha)vv^\top\right)v\nonumber\\
&=v^\top\adj\left(\mu_i(\beta+2)\right)v\\
&=c\det(\beta+2)v^\top (\beta+2)^{-1}\left(2v^\top(\beta_N+2)^{-1}w\right)^iww^\top v\nonumber\\
&=c(2v^\top(\beta_N+2)^{-1}w)^{i+1}.\nonumber
\end{align}
This is nonzero by assumption, so $\adj(\mu_{i+1})$ is nonzero. By Prop \ref{adj}, it is rank $1$. The matrix $\mu_{i+1}$ is symmetric and $w$ lies in its kernel:
\begin{equation}
w^\top\mu_{i+1}=w^\top\mu_i(\beta+2)+w^\top(\alpha+2)^ih_r(\alpha)vv^\top=0+0,
\end{equation}
so it is of the form $\adj(\mu_{i+1})=c(2v^\top(\beta_N+2)^{-1}w)^{i+1}ww^\top$. This completes the induction.

The final $\mu_l$ has determinant
\begin{align}
\det(\mu_l)&=\det(\beta+2)\left(\det(\mu_{l-1})+v^\top(\beta+2)^{-1}\adj(\mu_{l-1})(\alpha+2)^{l-1}h_r(\alpha)v\right)\nonumber\\
&=2\left(0+2^{l-1}c(v^\top(\beta_N+2)^{-1}w)^lw^\top(\alpha+2)^{l-1}h_r(\alpha)v\right)\\
&=c\left(2v^\top(\beta_N+2)^{-1}w\right)^l\left(w^\top(\alpha+2)^{l-1}h_r(\alpha)v\right).\nonumber
\end{align}
\end{proof}

\begin{thm}
\begin{align}
&d_{2k+1}((2k+1)n+1)=(2n+1)^k\\
&d_{2k}(2kn+1)=(-1)^{kn}\\
&d_{2k}(2kn+k+1)=-(-1)^{kn+{k\choose 2}}4^{k-1}(n+1)^{k-1}
\end{align}
\end{thm}
\begin{proof}
Given $w$, it is straightforward to verify the hypotheses and evaluate the final expression of Lemma \ref{main2} with the lemmas of Section \ref{struct}.

For the first formula, take $w$ to be
\begin{equation}
w_5=(-1)^n\left(\sum_{m=0}^n (-1)^me_{(2k+1)m}-\sum_{m=0}^{n-1} (-1)^me_{(2k+1)m+2k}\right),
\end{equation}
where $\{e_i\}_{i=0}^{N-1}$ is the standard basis.

For the second formula, $w$ to be
\begin{equation}
w_6=(-1)^n\left(\sum_{m=0}^n (-1)^me_{2km}-\sum_{m=0}^{n-1} (-1)^me_{2km+2k-1}\right).
\end{equation}

For the third formula, use
\begin{equation}
w_7=(-1)^{n-1}\left(\sum_{m=0}^n (-1)^me_{2km+k-1}-\sum_{m=0}^{n} (-1)^me_{2km+k}\right).
\end{equation}
\end{proof}

\section{Conjectures}

Let
\begin{equation}
d'_r(n)=\det\left(\frac{r}{2i+2j+r}{2i+2j+r\choose i+j}\right)_{i,j=0}^{n-1}.
\end{equation}
These sequences are considered alongside $d_r(n)$ in \cite{2018arXiv180105608C}. Computer experiments suggest the following conjectures:
\begin{align}
&d'_{2k+1}((2k+1)n)=d'_{2k+1}((2k+1)n+1)=(-1)^{kn},\\
&d'_{2k+1}((2k+1)n+k)=-d'_{2k+1}((2k+1)n+k+2)\\
&\phantom{d'_{2k+1}((2k+1)n+k)}=(-1)^{kn+\binom k 2}((2k+1)(n+1))^{k-1},\\
&d'_{2k+1}((2k+1)n+k+1)=0,\\
&d'_{2k}(kn)=-d'_{2k}(kn+1)=(-1)^{n\binom k 2}(n+1)^{k-1}.
\end{align}
Moreover, it seems that
\begin{align}
&\left(\frac{r}{2i+2j+r}{2i+2j+r\choose i+j}\right)_{i,j\ge 0}\\
&\hspace{60pt}=\begin{cases}
A(-a_k(\gamma)(\gamma+2)^{k-1})A^\top & \text{if } r=2k\\
A((-1)^{k+1}g_k(-\gamma)(\gamma-2)(\gamma+2)^{k-1})A^\top & \text{if } r=2k+1.
\end{cases}\nonumber
\end{align}

\bibliographystyle{plain}
\bibliography{Cigler-Tyson}

\end{document}